\documentclass[12pt]{cmip-l}

\newtheorem{theorem}{\bf Theorem}[section]
\newtheorem{lemma}[theorem]{\bf Lemma}
\newtheorem{remark}[theorem]{\bf Remark}

\newtheorem{proposition}[theorem]{\bf Proposition}
\newtheorem{corollary}[theorem]{\bf Corollary}
\newtheorem{definition}[theorem]{\bf Definition}
\newtheorem{example}[theorem]{\bf Example}

\newtheorem{problem}[theorem]{\bf Problem}

\newtheorem*{theor}{\bf Model Theorem}
\newtheorem*{conj}{\bf Conjecture}

\title[$\alpha$-stable coherent systems]{Existence of $\alpha$-stable coherent systems\\on algebraic curves}
\author{P. E. Newstead}
\address{Department of Mathematical Sciences\\University of Liverpool\\Peach Street\\Liverpool\\L69 7ZL\\UK}
\email{newstead@liv.ac.uk}
\date{3 September 2008}

\thanks{This survey is a much extended version of a lecture given at the Clay workshop held in Cambridge, Massachusetts, Oct 6--10, 2006. The survey began life earlier than that as a set of notes on a series of three lectures given in Liverpool in August 2005 at a meeting of the research group VBAC (Vector Bundles on Algebraic Curves) which formed part of the activity of the Marie Curie Training Site LIMITS. My thanks are due to the Clay Mathematics Institute for supporting me as a Senior Scholar and to the European Commission for its support of LIMITS (Contract No. HPMT-CT-2001-00277). My thanks are due also to Cristian Gonz\'alez-Mart\'{\i}nez, who wrote the first draft of this survey shortly after the VBAC meeting.}

\begin{document}\maketitle
\tableofcontents

\newpage This survey covers the theory of moduli spaces of $\alpha$-stable coherent systems on a smooth projective curve $C$ of genus $g\ge2$. (The cases $g=0$ and $g=1$ are special cases, for which we refer to \cite{LN1,LN2,LN3,LN4}.) I have concentrated on the issues of emptiness and non-emptiness of the moduli spaces. Other questions concerning irreducibility, smoothness, singularities etc.\ and the relationship with Brill-Noether theory are mentioned where relevant but play a subsidiary r\^ole. In the bibliography, I have tried to include a full list of all relevant papers which contain results on coherent systems on algebraic curves. In general, I have not included papers concerned solely with Brill-Noether theory except where they are needed for reference, but I have included some which have clear relevance for coherent systems but have not been fully developed in this context. For another survey of coherent systems, concentrating on structural results and including an appendix on the cases $g=0$ and $g=1$, see \cite{BL}; for a survey of higher rank Brill-Noether theory, see \cite{GT}.

We work throughout over an
algebraically closed field of characteristic 0. Many results are undoubtedly valid in finite characteristic, but we omit this possibility here since most of the basic papers on coherent systems assume characteristic $0$. As stated above, we assume that $C$ is a smooth projective curve of genus $g\ge2$. We denote by $K$ the canonical bundle on $C$.

\section{Definitions}\label{sec1}

Coherent systems are simply the analogues for higher rank of classical linear systems.

\begin{definition}{\em A {\em coherent system} on $C$ of type $(n,d,k)$ is
a pair $(E,V)$ where $E$ is a vector bundle of rank $n$ and degree
$d$ over $C$ and $V$ is a linear subspace of dimension $k$ of
$H^0(E)$. A {\em coherent subsystem} of $(E,V)$ is a pair $(F,W)$, where $F$ is a subbundle of $E$ and $W\subseteq V\cap H^0(F)$. We say that $(E,V)$ is {\em generated} ({\em generically generated}) if the canonical homomorphism $V\otimes{\mathcal O}\to E$ is surjective (has torsion cokernel).}
\end{definition}

Stability of coherent systems is defined with respect to a real parameter $\alpha$. We
define the $\alpha$-slope $\mu_\alpha(E,V)$ of a coherent system $(E,V)$ by
\begin{equation*}\mu_{\alpha}(E,V):=\frac{d}{n}+\alpha \frac{k}{n}.
\end{equation*}

\begin{definition}\begin{em}
$(E,V)$ is $\alpha$-{\em stable} ($\alpha$-{\em semistable}) if, for
every proper coherent subsystem $(F,W)$ of $(E,V)$,
\begin{equation*}\mu_{\alpha}(F,W)<(\leq )\ \mu_{\alpha}(E,V).
\end{equation*}
\end{em}\end{definition}

Every $\alpha$-semistable coherent system $(E,V)$ admits a Jordan-H\"older filtration
$$0=(E_0,V_0)\subset (E_1,V_1)\subset\ldots\subset(E_m,V_m)=(E,V)$$
such that all $(E_j,V_j)$ have the same $\alpha$-slope and all the quotients $(E_j,V_j)/(E_{j-1},V_{j-1})$ are defined and $\alpha$-stable. The associated graded object
$$\bigoplus_1^m(E_j,V_j)/(E_{j-1},V_{j-1})$$
is determined by $(E,V)$ and denoted by $\mbox{gr}(E,V)$. Two $\alpha$-semistable coherent systems $(E,V)$, $(E',V')$ are said to be {\it S-equivalent} if $$\mbox{gr}(E,V)\cong\mbox{gr}(E',V').$$

There exists a moduli space $G(\alpha ;n,d,k)$ for $\alpha$- stable coherent systems of type
$(n,d,k)$ \cite{KN,LeP,RV,BGMN}; this is a quasi-projective scheme and has a natural completion to a projective scheme $\widetilde{G}(\alpha ;n,d,k)$, whose points correspond to S-equivalence
classes of $\alpha$-semistable coherent systems of type
$(n,d,k)$. In particular, we have $G(\alpha;n,d,0)=M(n,d)$ and  $\widetilde{G}(\alpha;n,d,0)=\widetilde{M}(n,d)$, where $M(n,d)$ ($\widetilde{M}(n,d)$) is the usual moduli space of stable bundles (S-equivalence classes of semistable bundles).

We define also the {\em Brill-Noether loci}
\begin{equation*}B(n,d,k):=\{ E\in M(n,d): h^0(E) \geq k \}
\end{equation*}
and
\begin{equation*}\widetilde{B}(n,d,k):=\{ [E]\in \widetilde{M}(n,d): h^0(\mbox{gr} E) \geq k \},
\end{equation*}
where, for any semistable bundle $E$, $[E]$ denotes the S-equivalence class of  $E$ and $\mbox{gr}E$ is the graded object associated to $E$.

For $k\ge1$, we must have $\alpha>0$ for the existence of $\alpha$-stable coherent systems (for $\alpha\le0$, $(E,0)$ contradicts the $\alpha$-stability of $(E,V)$). Given this, the moduli space $G(\alpha;1,d,k)$ is independent of $\alpha$ and we denote it by $G(1,d,k)$; it coincides with the classical variety of linear series $G_d^{k-1}$.  

So suppose
$n\geq 2$ and $k\geq 1$. \begin{lemma}\label{lem1}If $G(\alpha ;n,d,k)\neq
\emptyset $ then
\begin{equation}\label{conditions1}\textnormal{$\alpha > 0$, $d>0$, $\alpha
(n-k)<d$.}\end{equation}\end{lemma}
\begin{proof}Let $(E,V)\in G(\alpha;n,d,k)$ with $n\ge2$, $k\ge1$. As already remarked, $(E,0)$ contradicts $\alpha$-stability for $\alpha
\leq 0$. If $(E,V)$ is generically generated, then certainly $d\ge0$ with equality if and only if $(E,V)\cong({\mathcal O}^n,H^0({\mathcal O}^n))$, which is not $\alpha$-stable for $n\ge2$. Otherwise, the subsystem $(F,V)$, where $F$ is the subbundle of $E$ generically generated by $V$, contradicts stability for $d\leq 0$.
Finally, let $W$ be a $1$-dimensional subspace of $V$ and let $L$ be the line subbundle of $E$ generically generated by $W$. Then $(L,W)$ contradicts $\alpha$-stability if $k<n$ and
$\alpha (n-k)\geq d$.\end{proof}

The range of permissible values of $\alpha$ is divided up by a finite set of critical values
\begin{equation*}0=\alpha_0 <\alpha_1 <\ldots \alpha_L
<\left\{\begin{array}{l@{\quad\quad}l@{\quad\quad}l}\frac{d}{n-k}
&\text{where $k<n$}\\
\infty &\text{where $k\geq n$,}
\end{array}\right.\end{equation*}
and $\alpha$-stability of $(E,V)$
cannot change within an interval $(\alpha_i ,\alpha_{i+1})$ (see \cite{BGMN}).
For $\alpha \in (\alpha_i ,\alpha_{i+1})$, we write 
$G_i(n,d,k):=G(\alpha;n,d,k)$. For $k\ge n$ and $\alpha > \alpha_L $ we write
$G_L(n,d,k):=G(\alpha ;n,d,k)$, with a similar definition for $k<n$ and $\alpha_L<\alpha<\frac{d}{n-k}$. The semistable moduli spaces $\widetilde{G}_i(n,d,k)$ are defined similarly.

\begin{lemma}\label{lem2}We have a morphism $\widetilde{G}_0(n,d,k)\rightarrow
\widetilde{B}(n,d,k)$, whose image contains $B(n,d,k)$. In fact, for any coherent system $(E,V)$,
\begin{itemize}
\item $E\mbox{ stable }\Rightarrow(E,V)\in G_0(n,d,k)$
\item $(E,V)\in G_0(n,d,k)\Rightarrow E\mbox{ semistable}$.
\end{itemize}\end{lemma}
\begin{proof}This follows easily from the definition of $\alpha$-stability.\end{proof}

\begin{definition}\begin{em}The {\em Brill-Noether number} $\beta(n,d,k)$ is defined by
$$\beta(n,d,k):=n^2(g-1)+1-k(k-d+n(g-1)).$$
\end{em}\end{definition}
This is completely analogous to the classical case ($n=1$) and is significant for two reasons. In the first place, we have

\begin{lemma}\label{lem3} {\em (\cite[Corollaire 3.14]{He}; \cite[Corollary 3.6]{BGMN})} Every irreducible component of $G(\alpha;n,d,k)$ has dimension $\ge\beta(n,d,k)$.
\end{lemma}

The second reason concerns the local structure of the moduli spaces. 

\begin{definition}\begin{em}Let $(E,V)$ be a coherent system. The {\em Petri map} of $(E,V)$ is the linear map
$$V\otimes H^0(E^*\otimes K)\longrightarrow H^0(\mbox{End}E\otimes K)$$
given by multiplication of sections.\end{em}\end{definition}

\begin{lemma}\label{lem4} {\em(\cite[Proposition 3.10]{BGMN})} Let $(E,V)\in G(\alpha;n,d,k)$. The following conditions are equivalent:
\begin{itemize}
\item $G(\alpha;n,d,k)$ is smooth of dimension $\beta(n,d,k)$ at $(E,V)$
\item the Petri map of $(E,V)$ is injective.
\end{itemize}\end{lemma}

\begin{definition}\begin{em}The curve $C$ is called a {\em Petri curve} if the Petri map of $(L,H^0(L))$ is injective for every line bundle $L$.\end{em}\end{definition}

It is a standard result of classical Brill-Noether theory that the general curve of any genus is a Petri curve \cite{Gie,Laz}. 

\section{The Existence Problem}\label{sec2}

Among the possible questions one can ask about the existence of $\alpha$-stable coherent systems are the following.
\begin{itemize}
\item[(i)] Given $n$, $d$, $k$, $\alpha$ satisfying
(\ref{conditions1}), is $G(\alpha ;n,d,k)\neq \emptyset $?
\item[(ii)] Given $n$, $d$, $k$, does there exist $(E,V)$ which is
$\alpha$-stable for all $\alpha $ satisfying (\ref{conditions1})?
\item[(iii)] Given $n$, $d$, $k$, does there exist $(E,V)$ with $E$ stable
which is $\alpha$-stable for all $\alpha$ satisfying
(\ref{conditions1})?
\end{itemize}
We can restate (ii) and (iii) in a nice way by defining
$$U(n,d,k):=\{(E,V)\in G_L(n,d,k)\mid E \mbox{ is stable}\}$$
and
$$U^s(n,d,k):=\{(E,V)\mid (E,V) \mbox{ is $\alpha$-stable for } \alpha>0, \alpha(n-k)<d\}.$$
Then (ii) and (iii) become
\begin{itemize}
\item[(ii)$'$] Given $n$, $d$, $k$ with $d>0$, is $U^s(n,d,k)\ne\emptyset$?
\item[(iii)$'$] Given $n$, $d$, $k$ with $d>0$, is $U(n,d,k)\ne\emptyset$?
\end{itemize}

The ``obvious'' conjecture is that the answers are affirmative if and only if $\beta(n,d,k)\ge0$. However, this is false in both directions; see, for example, Remarks \ref{rk1} and \ref{rk5}. A basic result, for which a complete proof has only been given very recently, is

\begin{theorem}\label{th1}For any fixed $n$, $k$,  $U(n,d,k)\neq \emptyset $ for all
sufficiently large $d$.
\end{theorem}
\begin{proof}For $k\leq n$, see \cite{BG2,BGMMN1}. For $k>n$, see \cite{Ba2} and \cite{Te3}.
\end{proof}

Now write 
$$d_0=d_0(n,k):=\min \{ d:G(\alpha ;n,d,k)\neq \emptyset 
\mbox{ for some $\alpha$ satisfying (\ref{conditions1}})\}.$$
The following would be an ``ideal'' result.
\begin{theor}For fixed $n$, $k$ with $n\geq 2$,
$k\geq 1$, 
\begin{itemize}\item[(a)] $G(\alpha ;n,d,k)\neq \emptyset $ if and
only if $\alpha>0$, $(n-k)\alpha<d$ and $d\ge d_0$ \item[(b)]
$B(n,d,k)\neq \emptyset$ if and only if $d \geq d_0$ \item[(c)] if
$d\geq d_0$, $U(n,d,k)\neq \emptyset$.
\end{itemize}
\end{theor}

Note that (c) is stronger than (a) and (b) combined.

One may also ask whether, for given $n$, $d$, there is an upper bound on $k$ for which $G(\alpha;n,d,k)\ne\emptyset$. In fact the existence of an upper bound on $h^0(E)$ for any $(E,V)\in G(\alpha;n,d,k)$ (possibly depending on $\alpha$) is used in the proof of existence of the moduli spaces (see, for example \cite{KN}). An explicit such bound is given in \cite[Lemmas 10.2 and 10.3]{BGMN}. The more restricted question of an upper bound on $k$ has now been solved by the following generalization of Clifford's Theorem.

\begin{theorem}\label{th2} {\em(\cite{LN5})}
Let $(E,V)$ be a coherent system of type $(n,d,k)$ with $k>0$ which is $\alpha$-semistable for some $\alpha>0$. Then
$$k\le\begin{cases}\frac{d}2+n&\text{if $0\le d\le 2gn$}\\d+n(1-g)&\text{if $d\ge2gn$}.\end{cases}$$
\end{theorem}

\section{Methods}\label{sec3}

We now describe some of the methods which can be used for tackling the existence problem.
\begin{itemize}
\item Degeneration Methods. These have been used very successfully by M.~Teixidor i Bigas \cite{Te1,Te2,Te3} (see section \ref{sec8}).
\item Extensions $0\rightarrow E_1 \rightarrow E \rightarrow E_2
\rightarrow 0$. For $k<n$, a special case, used in \cite{BG2,BGMMN1}, is
\begin{equation*}0 \rightarrow \mathcal{O}^k \rightarrow
E \rightarrow E_2 \rightarrow 0.
\end{equation*}
The method  is not
so useful when $V\not\subseteq H^0(E_1)$ because there is a problem of lifting sections of $E_2$ to $E$.
A more promising approach is to use extensions of coherent systems
\begin{equation*}0\rightarrow (E_1 ,V_1 )\rightarrow (E,V)
\rightarrow (E_2,V_2) \rightarrow 0,
\end{equation*} which are classified by $\mbox{Ext}^1((E_2,V_2),(E_1,V_1))$.
This approach was pioneered in \cite{BGMN} and has been used successfully in several papers \cite{BGMMN2,BBN,LN3}.
\item Syzygies and projective embeddings. Given a generated coherent system
$(E,V)$ of type $(n,d,k)$ with $k>n$, one can construct a morphism $C\rightarrow \mbox{Grass}(k-n,k)$. One can also use the sections of
$\det E$ to get a morphism from $C$ to a projective space. The relationship between these morphisms can be used to settle existence problems (see section \ref{sec7} and \cite{GMN}).
\item Elementary transformations. Consider an exact sequence
\begin{equation*}0\rightarrow E \rightarrow E' \rightarrow T
\rightarrow 0,
\end{equation*}where $T$ is a torsion sheaf. If $(E,V)$ is a coherent system, then so is $(E',V)$. This method has been used to construct $\alpha$-stable coherent systems for low (even negative) values of $\beta(n,d,k)$ (see sections \ref{sec6}, \ref{sec7} and \ref{sec8}).
\item Flips. Flips were pioneered by M.~Thaddeus \cite{Th} and introduced in this context in \cite{BGMN}. They are given by extensions of coherent systems and can allow results for one value of $\alpha$ to be transmitted to another value.
\item Dual Span. This was originally introduced by D.~C.~Butler \cite{Bu} and has been used very successfully in the case $k=n+1$ (see section \ref{sec6}). The idea is as follows. Let $(E,V)$ be a coherent system. We define a new coherent system $D(E,V)=(D_V(E),V')$, where $D_V(E)^*$ is the kernel of the evaluation map $V\otimes{\mathcal O}\to E$ and $V'$ is the image of $V^*$ in $H^0(D_V(E))$. In the case where $(E,V)$ is generated and $H^0(E^*)=0$, we have dual exact sequences
$$0\to D_V(E)^*\to V\otimes {\mathcal O}\to E\to 0$$
and
$$0\to E^*\to V^*\otimes{\mathcal O}\to D_V(E)\to 0$$
and $D(D(E,V))=(E,V)$, so this is a true duality operation. The main point here is that the stability properties of $D(E,V)$ are closely related to those of $(E,V)$ \cite[section 5.4]{BGMN}.
\item Covering Spaces. Suppose that $f:Y\rightarrow C$ is a covering (maybe ramified). If $(E,V)$ is a coherent system on $Y$, then $(f_*E,V)$ is a coherent system on $C$. One can for example take $E$ to be a line bundle. The rank of $f_*E$ is then equal to the degree of the covering and $\deg E$ is easy to compute. The problem is to prove $\alpha$-stability. Preliminary work suggests that this may be interesting, but to get really good results one needs to take into account the fact that $Y$ is not a general curve and may possess line bundles with more than the expected number of independent sections.
\item Homological methods. In classical Brill-Noether theory, homological methods have been very successful in proving non-emptiness of Brill-Noether loci; essentially one views the loci as degeneration loci and uses the Porteous formula \cite[II (4.2) and VII (4.4)]{ACGH}. It is not trivial to generalize this to higher rank Brill-Noether loci because of non-compactness of the moduli space of stable bundles (when $\gcd(n,d)\ne1$) but more particularly because the cohomology is much more complicated. There is some unpublished work in a special case by Seongsuk Park \cite{Park} (see also section \ref{sec9}). Once one knows that $B(n,d,k)\ne\emptyset$, one gets also $G_0(n,d,k)\ne\emptyset$ by Lemma \ref{lem2}.
\item Gauge theory. This has been used for constructing the moduli spaces of coherent systems \cite{BDGW,BG1}, but not (so far as I know) for proving they are non-empty.
\end{itemize}

\section{$0<k\leq n$}\label{sec4}

In this case the existence problem is completely solved and there are some results on irreducibility and smoothness.

\begin{theorem}\label{th3}{\em (\cite[Theorem 3.3]{BGMMN1})} Suppose $0<k\leq n$. Then $G(\alpha ;n,d,k)\neq
\emptyset $ if and only if (\ref{conditions1}) holds and in
addition
\begin{equation}\label{conditions2}\textnormal{$k\leq
n+\frac{1}{g}(d-n),\ \ \ (n,d,k)\neq
(n,n,n)$.}\end{equation}
Moreover, if (\ref{conditions2}) holds and $d>0$, then $U(n,d,k)\neq
\emptyset$ and is smooth and irreducible.
\end{theorem}
\begin{proof}[Indication of proof]It is known \cite[Theorems 5.4 and 5.6]{BGMN} that
$G_L(n,d,k) \neq \emptyset $ under the stated conditions. Moreover $G_L(n,d,k)$ is smooth and irreducible and every
element has the form
\begin{align*}& 0\rightarrow \mathcal{O}^k \rightarrow E
\rightarrow F \rightarrow 0\ \  (k<n),\\& 0\rightarrow \mathcal{O}^n
\rightarrow E \rightarrow T \rightarrow 0\ \  (k=n),
\end{align*}with $F$ semistable, $T$ torsion. Necessity of (\ref{conditions2}) is proved in \cite[Corollary 2.5 and Lemma 2.10]{BGMMN1} and \cite[Remark 5.7]{BGMN}. In both cases, given (\ref{conditions2}) and $d>0$, the general $E$ of this form is stable; this is a consequence of \cite[Th\'eor\`eme A.5]{Mer2} (see also \cite{BGN,BMNO,Mer1}).\end{proof}
\begin{remark}\label{rk1}\begin{em} Conditions (\ref{conditions2}) are strictly stronger than $\beta(n,d,k)\ge0$.\end{em}\end{remark}
\begin{corollary}\label{cor1}Model Theorem holds with
$$d_0=\begin{cases}\max
\{ 1, n-g(n-k) \}
&(k<n)\\
n+1 &(k= n).\end{cases}$$\end{corollary}

\section{$0<d\leq 2n$}\label{sec5}

In this case the existence problem is completely solved over any curve and all non-empty moduli spaces are irreducible.

\begin{theorem}\label{th4}{\em(\cite[Theorem 5.4]{BGMMN2})} Suppose $g\geq 3$ and $C$ is not hyperelliptic. If
$0<d\leq 2n$, then $G(\alpha ;n,d,k)\neq \emptyset $ if and only
if either (\ref{conditions1}) and (\ref{conditions2}) hold or
$(n,d,k)=(g-1,2g-2,g)$. Whenever it is non-empty, $G(\alpha ;n,d,k)$ is irreducible. Moreover $U(n,d,k)\ne\emptyset$ and is smooth.
\end{theorem}
\begin{proof}[Indication of proof]The necessity of the
conditions for $0<d<2n$ follows from \cite{Mer1}. For $d=2n$, further calculations are necessary. For $k\le n$, the sufficiency of (\ref{conditions1}) and (\ref{conditions2}) has already been proved in Theorem \ref{th3}. For $k>n$, one requires the results of \cite{BGN,Mer1,Mer3} and a lemma stating that under these conditions $B(n,d,k)\ne\emptyset\Rightarrow U(n,d,k)\ne\emptyset$. For $(n,d,k)=(g-1,2g-2,g)$ we take $(E,V)=D(K,H^0(K))$ (see section \ref{sec6}).\end{proof}

\begin{remark}\label{rk2}\begin{em}For $C$ hyperelliptic, the theorem remains true for $0<d<2n$, but
some modification is needed for $d=2n$ \cite[Theorem 5.5]{BGMMN2}.\end{em}
\end{remark}

The reference \cite{BGMMN2} includes an example to show that these nice results do not extend beyond $d=2n$.

\begin{example}\label{ex}\begin{em}(\cite[section 7]{BGMMN2}) Let $(E,V)$ be a coherent system of type $(n,d,k)$ with
\begin{equation}\label{eqn11}n+\frac{1}{g}(d-n)<k<\frac{ng}{g-1}.
\end{equation}Then $(E,V)$ is not $\alpha$-semistable for large
$\alpha$. Moreover, if
$C$ is non-hyperelliptic and $3\leq r\leq g-1$, there exists $(E,V)$ of
type $(rg-r+1,2rg-2r+3,rg+1)$ with $E$ stable, and (\ref{eqn11}) holds in this case.\end{em}
\end{example}

\section{$k=n+1$}\label{sec6}

In this case the existence problem is almost completely solved when $C$ is a Petri curve, which we assume until further notice. Full details are contained in \cite{BBN}.

We use the dual span construction for $(L,V)$, where $L$ is a line bundle of degree $d>0$, $(L,V)$ is generated and $\dim V=n+1$.

\begin{lemma}\label{lem5}{\em (\cite[Corollary 5.10]{BGMN})} $D(L,V)$ is $\alpha$-stable for large $\alpha$.
\end{lemma}
\begin{proof}Let $(F,W)$ be a coherent subsystem of $D(L,V)$ with
$\mbox{rk} F<n$. Since $D(L,V)$ is generated and $H^0(D_V(L)^*)=0$, it follows that $D_V(L)/F$ is non-trivial and that it is generated by the image of  $V^*$ in $H^0(D_V(L)/F)$. This implies that this image has dimension
$\geq n-\mbox{rk} F+1$. So $\frac{\dim W}{\text{rk}F}\leq
1<\frac{n+1}{n}$ which implies the result.
\end{proof}
\begin{theorem}\label{th5}$G_L(n,d,n+1)$ is birational to $G(1,d,n+1)=G_d^n$.
\end{theorem}
\begin{proof}Note that $\beta (n,d,n+1)=\beta (1,d,n+1)=\dim
(G(1,d,n+1))$. Now, do a parameter count to show that
non-generated $(L,V)$ and $(E,V^*)$ contribute $<\beta(n,d,n+1)$ to
the dimension (see the proof of \cite[Theorem 5.11]{BGMN} for details).
\end{proof}
\begin{corollary}\label{cor2}$G_L(n,d,n+1)\neq \emptyset $ if and only if $\beta(n,d,n+1) \geq 0$,
i.e. if and only if
\begin{equation}\label{conditions3}d\geq
g+n-\Big{[}\frac{g}{n+1}\Big{]}.
\end{equation}
\end{corollary}
\begin{proof}This follows from classical Brill-Noether theory, since $C$ is Petri.\end{proof}
\begin{theorem}\label{cor3}{\em(\cite[Theorem 2]{Br})} If (\ref{conditions3}) holds and $d\leq g+n$, then
$U^s(n,d,n+1)\neq \emptyset$. Moreover, except for $g=n=2$, $d=4$, $U(n,d,n+1)\ne\emptyset$\end{theorem}
\begin{proof}If (\ref{conditions3}) holds and $d\leq g+n$, there exists a line bundle
$L$ of degree $d$ with $h^0(L)=n+1$, so we can take $V=H^0(L)$ for
such $L$, and suppose that $V$ generates $L$. By Lemma 6.1, $D(L,V)$ is $\alpha$-stable for large $\alpha$, while, by \cite[Theorem 2]{Bu} (see also \cite[Proposition 4.1]{Br}), $D_L(V)$ is stable for general $L$. The result for $g=n=2$, $d=4$ is contained in \cite[Proposition 4.1]{Br}.
\end{proof}
\begin{theorem}\label{th6}{\em (\cite[Theorem 2]{Br})}
If $G(\alpha ;n,d,n+1)\neq \emptyset $ for some
$\alpha$, then (\ref{conditions3}) holds.
\end{theorem}
\begin{proof}[Indication of proof]One can show that, for $d\le g+n$,
$$G(\alpha;n,d,n+1) =G_L(n,d,n+1)$$ for all $\alpha >0$. Now use Corollary \ref{cor2}.
\end{proof}
\begin{corollary}\label{cor4}$d_0(n,n+1)=g+n-\Big{[}\frac{g}{n+1}\Big{]}$.
\end{corollary}

\begin{theorem}\label{th7}{\em (\cite[Theorem 3]{Br})} If $g\ge n^2-1$, then Model Theorem holds.\end{theorem}

\begin{theorem}\label{th8}{\em(\cite[Theorems 7.1, 7.2, 7.3]{BBN})} Model Theorem holds for $n=2, 3, 4$ and $g\ge3$.
\end{theorem}

For $g=2$, Model Theorem does not quite hold; the result is as follows.

\begin{theorem}\label{th9}{\em(\cite[Theorem 8.2]{BBN})} Let $X$ be a curve of genus $2$. Then $d_0=n+2$ and
\begin{itemize}
\item $U^s(n,d,n+1)\ne\emptyset$ if and only if $d\ge d_0$
\item $U(n,d,n+1)\ne\emptyset$ if and only if $d\ge d_0$, $d\ne2n$.
\end{itemize}
\end{theorem}

The proofs of these theorems depend on combining several techniques including those of \cite{Br}, extensions of coherent systems and the following result.
\begin{proposition}\label{th10}Suppose that $d\ge d_1$, where 
\begin{equation*}d_1:=
\begin{cases}\frac{n(g+3)}{2}+1
&\text{if $g$ is odd}\\
\frac{n(g+4)}2+1&\text{if $g$ is even and $n>\frac{g!}{\left(\frac{g}2\right)!\left(\frac{g}2+1\right)!}$}
\\
\frac{n(g+2)}{2}+1 &\text{if $g$ is even and $n\leq
\frac{g!}{\left(\frac{g}2\right)!\left(\frac{g}2+1\right)!}$}.
\end{cases}\end{equation*}
Then $U(n,d,n+1)\ne\emptyset$.\end{proposition}
\begin{proof}This is proved using elementary transformations (for details see \cite[Proposition 6.6]{BBN}). The restriction on $n$ in the third formula is required because we need to have $n$ non-isomorphic line bundles $L_i$ of degree $\frac{d_1-1}n$ with $h^0(L_i)=2$. For the number of such line bundles in this case, see \cite[V (1.2)]{ACGH}. \end{proof}

\begin{remark}\label{rk3}\begin{em} Although Model Theorem is not established in all cases, one can say that $U(n,d,n+1)$ is always smooth and is irreducible of dimension $\beta(n,d,n+1)$ whenever it is non-empty and $\beta(n,d,n+1)>0$ \cite[Remark 6.2]{BBN}.
\end{em}\end{remark}

Now let us replace the Petri condition by the condition that $C$ be general (in some unspecified sense). Teixidor's result (see Theorem \ref{th15}) takes the following form when $k=n+1$.

\begin{theorem}\label{th11} Suppose that $C$ is a general curve of genus $g$ and that $d\ge d_1'$, where
\begin{equation*}d_1'=
\begin{cases}\frac{n(g+1)}{2}+1
&\text{if $g$ is odd}\\
\frac{n(g+2)}{2}+1 &\text{if $g$ is even}.
\end{cases}\end{equation*}
Then $U(n,d,n+1)\ne\emptyset$.
\end{theorem}

Using this, we can prove
\begin{theorem}\label{th12} Suppose that $C$ is a general curve of genus $3$. Then Model Theorem holds.
\end{theorem}
\begin{proof} The methods of \cite{BBN} are sufficient to prove (for any Petri curve of genus $3$) that $U(n,d,n+1)\ne\emptyset$ if $d\ge d_0$ and $d\ne2n+2$ (see \cite[Theorem 8.3]{BBN}). The exceptional case is covered by Theorem \ref{th11}.
\end{proof}

\begin{remark}\begin{em}
Suppose now that $C$ is any smooth curve of genus $g$. Since $\gcd(n,d,n+1)=1$, a specialization argument shows that, if $G(\alpha;n,d,n+1)\ne\emptyset$ on a general curve and $\alpha$ is not a critical value, then $G(\alpha;n,d,n+1)\ne\emptyset$ on $C$. A priori, this does not imply that $U(n,d,n+1)\ne\emptyset$, but Ballico \cite{Ba4} has used Teixidor's result to show that, when $d\ge d_1'$, we have $U^s(n,d,n+1)\ne\emptyset$. If $\gcd(n,d)=1$, this gives $U(n,d,n+1)\ne\emptyset$.\end{em}\end{remark}

\section{$n=2$, $k=4$}\label{sec7}

This is the first case in which we do not know the value of $d_0$ even on a general curve. Let us define\begin{equation}\label{conditions4}d_2:=
\left\{\begin{array}{l@{\quad\quad}l@{\quad\quad}l}g+3
&\text{if $g$ is even}\\
g+4 &\text{if $g$ is odd}.
\end{array}\right.\end{equation}
Note that $\beta(2,d,4)=4d-4g-11$, so $\beta(2,g+3,4)=1$ and $g+3$ is the smallest value of $d$ for which $\beta(2,d,4)\ge0$. Moreover, on a Petri curve, $\frac{d_2-1}2$ is the smallest degree for which there exists a line bundle $L$ with $h^0(L)\ge2$.

Teixidor has proved the following result by degeneration methods.
\begin{theorem}\label{th13} {\em\cite{Te1}} Let $C$ be a general curve of genus $g\ge3$. If $d\ge d_2$, then $U(2,d,4)\ne\emptyset$.
\end{theorem}

By completely different methods, we can prove the following stronger result (further details will appear in \cite{GMN}).

\begin{theorem}\label{th14}Let $C$ be a Petri curve of genus $g\geq 3$. Then
\begin{enumerate}\item[(i)] $U(2,d,4)\neq \emptyset $ for $d\geq d_2$,
\item[(ii)] if $d< d_2-1$ and $(E,V)$ is $\alpha$-semistable
of type $(2,d,4)$ for some $\alpha$, then $(E,V)\in U(2,d,4)$.
Moreover, $U(2,d',4)\neq \emptyset$ for $d\leq d'\leq d_2-1$,
\item[(iii)] $G(\alpha ;2,d,4)=\emptyset $ for all $\alpha$ for
$d< min \{ \frac{4g}{5}+4 ,d_2-1 \}$.
\end{enumerate}
\end{theorem}

\begin{proof} (i) Consider
\begin{equation*}0\rightarrow L_1\oplus L_2 \rightarrow E
\rightarrow T \rightarrow 0,\end{equation*}where the $L_i$ are line bundles with
$\deg(L_i)=\frac{d_2-1}{2}$, $h^0(L_i)=2$, $L_1\not\cong L_2$ and $T$ is a torsion sheaf of length $d-d_2+1$. By classical Brill-Noether theory, such line bundles always exist on a Petri curve.
By \cite[Th\'eor\`eme A.5]{Mer2}, the general such $E$ is stable. Now, it is easy to show
that $(E,V)\in U(2,d,4)$ where $V=H^0(L_1)\oplus H^0(L_2)$.

(ii) Let $(F,W)$ be a coherent subsystem of $(E,V)$ with
$\mbox{rk}F=1$, where $W=V\cap H^0(F)$. Suppose $E$ is not stable and
choose $F$ with $\deg F\geq \frac{d}{2}$. Then $\deg(E/F)\leq
\frac{d}{2}< \frac{d_2-1}{2}$. Hence, by classical Brill-Noether theory, $h^0(E/F)\leq 1$, so $\dim W\geq 3$.
This contradicts $\alpha$-semistability for all $\alpha$.

So $E$ is stable. Now, for any $(F,W)$, $\deg F<\frac{d}{2}$, so
$h^0(F)\leq 1$, hence $\dim W\leq 1$. So
\begin{equation*}\mu_{\alpha}(F,W)<\frac{d}{2}+\alpha
<\frac{d+4\alpha}{2},\end{equation*}for all $\alpha >0$, so
$(E,V)\in U(2,d,4)$.

For the last part, we proceed by induction; we need to prove that, if $d< d_2-1$, then
$$U(2,d,4)\ne\emptyset\Rightarrow U(2,d+1,4)\ne\emptyset.$$
So suppose $(E,V)\in U(2,d,4)$ and consider an elementary transformation
$$0\rightarrow  E\rightarrow E'\rightarrow T\rightarrow 0,$$
where $T$ is a torsion sheaf of length $1$.  Then $(E',V)\in U(2,d+1,4)$ if and only if $E'$ is stable. It is easy to see that $E'$ is semistable. If $E'$ is strictly semistable, then it possesses a line subbundle $L$ of degree $\frac{d+1}2$. Since $E$ is stable, $L\cap E$ must have degree $\frac{d-1}2$ and so, by classical Brill-Noether theory, $\dim(H^0(L)\cap V)\le1$. Hence $h^0(E/L)\ge3$, which is a contradiction since $\deg(E/L)=\frac{d+1}2\le\frac{d_2-1}2$. 

(iii) If $(E,V)\in G(\alpha ;2,d,4)$ with $d< min \{ \frac{4g}{5}+4 ,d_2-1 \}$, we have $(E,V)\in U(2,d,4)$ by (ii). It is easy to check that $(E,V)$ is generically generated and the proof of (ii) shows that the subsheaf $E'$ of $E$ generated by
$V$ is stable. We have an exact sequence
\begin{equation*}0\rightarrow D_V(E')^{\ast} \rightarrow V\otimes
\mathcal{O} \rightarrow E' \rightarrow 0,
\end{equation*}
which induces two further exact sequences 
\begin{equation}\label{eqn5}0\rightarrow N \rightarrow \wedge^2 V\otimes
\mathcal{O} \rightarrow \det E' \rightarrow 0\end{equation} and
\begin{equation}\label{eqn6}0\rightarrow (\det E')^{\ast} \rightarrow N \rightarrow E'\otimes D_V(E')^{\ast} \rightarrow 0.\end{equation}
Now $D_V(E')$ is stable of the same slope as $E'$, so, by (\ref{eqn6}), $$h^0(N)\le h^0(E'\otimes D_V(E')^{\ast})\leq 1.$$
So
$h^0(\det E')\geq 5$
by (\ref{eqn5}). Hence, by classical Brill-Noether theory,
\begin{equation*}\deg E\geq \deg E'\geq \frac{4g}{5}+4.
\end{equation*}\end{proof}

\begin{corollary}\label{cor5}If $\frac{4g}{5}+4>d_2-1$, then $d_0=d_2$. If
$\frac{4g}{5}+4 \leq d_2-1$, then $\frac{4g}{5}+4\leq d_0 \leq
d_2$. Moreover, in all cases, Model Theorem holds.
\end{corollary}

\begin{proof}
The second part follows immediately from the theorem. If $\frac{4g}5+4>d_2-1$, then certainly $d_2-1\le d_0\le d_2$. In fact, if $(E,V)\in G(\alpha;2,d_2-1,4)$, the argument used to prove part (iii) of the theorem shows that $\deg E\ge\frac{4g}5+4$, which is a contradiction.

Model Theorem is now clear except possibly when $d_0=d_2-1$. But in this case, there exists $(E,V)\in G(\alpha;2,d_2-1,4)$ for some $\alpha$. If $E$ is not stable, then, in the proof of part (ii) of the theorem, we obtain $h^0(E/F)\le2$, so $\dim W\ge2$, contradicting $\alpha$-stability. The rest of the proof works to show that $(E,V)\in U(2,d_2-1,4)$.\end{proof}

\begin{remark}\begin{em}
Theorems \ref{th13} and \ref{th14}(i) and the first statement of Corollary \ref{cor5} fail for $g=2$. In this case $d_2=5$  and the proof of Theorem \ref{th14}(i) fails because there is only one line bundle $L$ of degree $2$ with $h^0(L)=1$. Moreover it follows from Riemann-Roch that any stable bundle $E$ of rank $2$ and degree $5$ has $h^0(E)=3$ and one can easily deduce that $G(\alpha;2,5,4)=\emptyset$ for all $\alpha>0$. In fact $d_0=6$ and Model Theorem holds (see statement (3) preceding Lemma 6.6 in \cite{Br}).\end{em}\end{remark}
We finish this section with an example

\begin{example}\begin{em}\label{ex2} If $4\le g\le8$ and $g$ is even, Corollary \ref{cor5} gives $d_0=d_2=g+3$.  Suppose now that $g=10$. Then $\frac{4g}{5}+4=12=g+2=d_2-1$, so $d_0=12$ or
$13$.

Suppose $(E,V)\in G(\alpha;2,12,4)$. The argument in the proof of Theorem \ref{th14}(ii)/(iii)
shows that $E'$ and $D(E')$ are stable and $h^0(\det E')\geq 5$.
By classical Brill-Noether theory, $\deg E'\geq 12$, so $E'=E$. Also $h^0(\det E)=5$.

Write $L=\det E$. We can choose a $3$-dimensional subspace $W$ of
$H^0(L)$ such that there is an exact sequence
\begin{equation*}0\rightarrow E^{\ast} \rightarrow W\otimes
\mathcal{O} \rightarrow L \rightarrow 0,
\end{equation*}which induces
\begin{equation*}0\rightarrow H^0(L\otimes E^{\ast}) \rightarrow W\otimes
H^0(L) \stackrel{\psi}{\rightarrow} H^0(L^{\otimes 2}) \rightarrow 0.
\end{equation*} Now $L\otimes E^{\ast}\cong E$, so
$h^0(L\otimes E^*)\ge4$. Hence $\dim\mbox{Ker}\,\psi\ge4$, from which it follows that the linear map 
\begin{equation}\label{eqn7}S^2(H^0(L))\rightarrow H^0(L^{\otimes 2})
\end{equation} 
is not injective. Both of the spaces in (\ref{eqn7}) have dimension 15, so (\ref{eqn7}) is not surjective. On the other hand, Voisin has proved that, for general $C$, (\ref{eqn7}) is surjective \cite{Voi}.
So, for general $C$, $G(\alpha;2,12,4)=\emptyset$ and $d_0=13$.

If (\ref{eqn7}) is not surjective, the sections of $L$ determine an embedding $C\hookrightarrow
\mathbb{P}^4$ whose image lies on the quadric $q$ whose equation generates the kernel of (\ref{eqn7}).

If  $q$ has rank $5$, it is easy to check that $(E,V)$ exists. This situation does in fact arise for certain Petri curves lying on K3 surfaces (see \cite{Voi}), so there exist Petri curves for which $d_0=12$. If $q$ has rank $4$, it is possible that $(E,V)$ is strictly $\alpha$-semistable for all $\alpha>0$, so we can make no deduction about $d_0$. The quadric $q$ cannot have rank $<4$.

Further details of this example will appear in \cite{GMN}.\end{em}\end{example}

\section{$k\ge n+2$}\label{sec8}

The situation in section \ref{sec7} is typical of that for $k\ge n+2$, except that in general we know even less. Until recently, no reasonable bound for $d_0$ had been established. This has now been rectified by Teixidor, who has obtained a bound for general $C$ by degeneration methods. We state her result in a slightly adapted version.

\begin{theorem}\label{th15}{\em(\cite{Te3})} Let $C$ be a general curve of genus $g\ge2$. If $k>n$, then $U(n,d,k)\ne\emptyset$ for $d\ge d_3$, where
\begin{equation}\label{eqn8}d_3:=
\begin{cases}k+n(g-1)-n\left[\frac{g-1}{[k/n]}\right]+1
&\text{if $n\mid k$}\\
k+n(g-1)-n\left[\frac{g-1}{[k/n]+1}\right] &\text{if $n\!\not\,\mid k$}.
\end{cases}\end{equation}
Moreover $U(n,d,k)$ has a component of dimension $\beta(n,d,k)$.
\end{theorem}

It should be noted that this is not best possible; in particular, when $n=2$, $k=4$, it is weaker than Theorem \ref{th13}. Also, Teixidor states her result in terms of non-emptiness of all $G(\alpha;n,d,k)$, but her proof gives the stronger statement that $U(n,d,k)\ne\emptyset$. For other formulations of (\ref{eqn8}), see \cite{Te4, Mer2}, where the corresponding result for Brill-Noether loci is proved.

For an arbitrary smooth curve, we have the following theorem of Ballico \cite{Ba4}, obtained from Theorem \ref{th15} by a specialization argument. 

\begin{theorem}\label{th16}
Let $C$ be a smooth curve of genus $g$ and $k>n$. If $\gcd(n,d,k)=1$ and $d\ge d_3$, then $U^s(n,d,k)\ne\emptyset$. If moreover $\gcd(n,d)=1$, then $U(n,d,k)\ne\emptyset$.
\end{theorem}

When $n\mid k$, we can prove the following theorem  without any coprimality assumptions by using the methods of \cite{Mer2}.

\begin{theorem}\label{th17} Let $C$ be a smooth curve of genus $g\ge2$. If $k>n$, $n\mid k$, then $U(n,d,k)\ne \emptyset$ for $d\ge d_3$. If in addition $C$ is Petri and $k':=\frac{k}n\mid g$, then $U(n,d,k)\ne \emptyset$ for $d\ge d_3-n$, provided that
$$n\le n':= g!\prod^{k'-1}_{i=0}\frac{i!}{\left(i+\frac{g}{k'}\right)!}.$$
\end{theorem}
\begin{proof} Note that $d_3-1$ is divisible by $n$ and $\beta\left(1,\frac{d_3-1}n,\frac{k}n\right)>0$. Hence, by classical Brill-Noether theory, we can find $n$ non-isomorphic line bundles $L_i$ with $\deg L_i=\frac{d_3-1}n$ and $h^0(L_i)\ge\frac{k}n$. Choose subspaces $V_i$ of $H^0(L_i)$ of dimension $\frac{k}n$ and consider an exact sequence
\begin{equation}\label{eqn9}
0\rightarrow L_1\oplus \ldots \oplus L_n\rightarrow E\rightarrow T\rightarrow 0,
\end{equation}
where $T$ is a torsion sheaf of length $d-d_3+1>0$. Let $V:=\bigoplus H^0(L_i)\subseteq H^0(E)$. Then $(L_1\oplus\ldots\oplus L_n,V)$ is $\alpha$-semistable for all $\alpha>0$. If $E$ is stable, it follows that $(E,V)\in U(n,d,k)$. On the other hand, the general $E$ given by (\ref{eqn9}) is stable by \cite[Th\'eor\`eme A.5]{Mer2}. The first result follows.

Now suppose $k'\mid g$. Then $\left[\frac{g-1}{k'}\right]=\frac{ng}k-1$. Let
$$d':=\frac{d_3-1}n-1=k'+g-1-\frac{g}{k'}$$
and $\beta(1,d',k')$=0. So, again by classical Brill-Noether theory, there exist line bundles $L_i$ with $\deg L_i=d'$ and $h^0(L_i)\ge k'$, and, if $C$ is Petri, there are precisely $n'$ such line bundles up to isomorphism \cite[V (1.2)]{ACGH}. The proof now goes through as before.
\end{proof}

\begin{remark}\label{rk5}\begin{em} It is worth noting that, when $n\mid k$ and $k'\mid g$, 
$$\beta(n,d_3-n,k)=k-n^2+1,$$ so one can have $U(n,d,k)\ne\emptyset$, even on a Petri curve, with negative Brill-Noether number.
\end{em}\end{remark}

For a lower bound on $d_0$, we have the following result of Brambila-Paz.

\begin{theorem}\label{th18}{\em(\cite[Theorem 1]{Br})} Let $C$ be a Petri curve of genus $g$. Suppose $k>n$ and $\beta(n,d,n+1)<0$. Then $G(\alpha;n,d,k)=\emptyset$ for all $\alpha>0$.
\end{theorem}

Combining Theorems \ref{th15} and \ref{th18}, we obtain that, on a general curve,
$$n+g-\left[\frac{g}{n+1}\right]\le d_0(n,k)\le d_3.$$

\section{$n=2$, $\det E=K$}\label{sec9}

In this section we consider coherent systems $(E,V)$ with $\mbox{rk}\,E=2$, $\det E\cong K$. These have not to my knowledge been studied in their own right, but the corresponding problem for Brill-Noether loci has attracted quite a lot of attention and this has some implications for coherent systems. In what follows, we consistently denote the corresponding spaces by $B(2,K,k)$, $G(\alpha;2,K,k)$ etc.

The first point to note is that some definitions have to be changed. There is a canonical skew-symmetric isomorphism between $E$ and $E^*\otimes K$. As a result of this, the Petri map has to be replaced by the {\em symmetrized Petri map}
\begin{equation}\label{eqn10}
S^2H^0(E)\rightarrow H^0(S^2E).
\end{equation}
This map governs the infinitesimal behaviour of $B(2,K,k)$. Moreover, the Brill-Noether number giving the expected dimension of $B(2,K,k)$ and $G(\alpha;2,K,k)$ must be replaced by
$$\beta(2,K,k):=3g-3-\frac{k(k+1)}2.$$
All components of $B(2,K,k)$ have dimension $\ge\beta(2,K,k)$. The intriguing thing is that often $\beta(2,K,k)>\beta(2,2g-2,k)$, so that the expected dimension of $B(2,K,k)$ is greater than that of $B(2,2g-2,k)$, although the latter contains the former!

These Brill-Noether loci were studied some time ago by Bertram and Feinberg \cite{BF} and by Mukai \cite{Mu}, who independently proposed the following conjecture.

\begin{conj} Let $C$ be a general curve of genus $g$. Then $B(2,K,k)$ is non-empty if and only if $\beta(2,K,k)\ge0$. 
\end{conj}

This was verified by Bertram and Feinberg and Mukai for low values of $g$. Of course, if $B(2,K,k)\ne\emptyset$, then $G_0(2,K,k)\ne\emptyset$, but we cannot make deductions about $U(2,K,k)$ without further information.

Some of these results have been extended and there are also some general existence results due to Teixidor.

\begin{theorem}\label{th19}{\em(\cite{Te5})} Let $C$ be a general curve of genus $g$. If either $k=2k_1$ and $g\ge k_1^2\ (k_1>2)$, $g\ge 5\ (k_1=2)$, $g\ge3\ (k_1=1)$ or $k=2k_1+1$ and $g\ge k_1^2+k_1+1$, then $B(2,K,k)\ne\emptyset$ and has a component of dimension $\beta(2,K,k)$.
\end{theorem}

\begin{corollary}\label{cor6} Under the conditions of Theorem \ref{th19}, $G_0(2,K,k)\ne\emptyset$.
\end{corollary}

The most significant result to date is also due to Teixidor.

\begin{theorem}\label{th20}{\em(\cite{Te6})} Let $C$ be a general curve of genus $g$.
Then the symmetrized Petri map (\ref{eqn10}) is injective for every semistable bundle $E$. 
\end{theorem}

\begin{corollary}\label{cor7} Let $C$ be a general curve of genus $g$. If $\beta(2,K,k)<0$, then $G_0(2,K,k)=\emptyset$.
\end{corollary}

To prove the conjecture, it is therefore sufficient to extend Theorem \ref{th19} to all cases where $\beta(2,K,k)\ge0$. Progress on this has been made recently by Seongsuk Park using homological methods \cite{Park}.

\begin{remark}\label{rk6}\begin{em} Theorem \ref{th20} and Corollary \ref{cor7} fail for certain Petri curves lying on K3 surfaces. In fact, based on work of Mukai, Voisin \cite[Th\'eor\`eme 0.1]{Voi} observed that on these curves $B\left(2,K,\frac{g}2+2\right)\ne\emptyset$ for any even $g$. However, for $g$ even, $g\ge10$, the Brill-Noether number $\beta\left(2,K,\frac{g}2+2\right)<0$.\end{em}\end{remark}

\section{Special curves}\label{sec10}

The results of sections \ref{sec3} and \ref{sec4} are valid for arbitrary smooth curves. On the other hand, while some of the results of the remaining sections (dealing with the case $k>n$) are valid on arbitrary smooth curves, most of them hold on general curves. Indeed, if $k>n$, the value of $d_0$ certainly depends on the geometry of $C$ and not just on $g$. This is of course no different from what happens for $n=1$, but it is already clear that the distinctions in higher rank are more subtle than those for the classical case (see, for instance, Example \ref{ex2} and Remark \ref{rk6}). 

There is as yet little work in the literature concerning coherent systems on special curves; here by {\em special} we mean a curve for which the moduli spaces of coherent systems exhibit behaviour different from that on a general curve. So far as I am aware, the only papers relating specifically to special curves are those of Ballico on bielliptic curves \cite{Ba1} and Brambila-Paz and Ortega on curves with specified Clifford index \cite{BO}. In the latter paper, the authors define positive numbers
$$d_u=d_u(n,g,\gamma):=\left\{\begin{array}{l@{\quad\quad}l@{\quad\quad}l}n+g-1+\frac{g-1}{n-1}
&\text{if $\gamma\ge g-n$}\\
2n+\gamma+\frac\gamma{n-1} &\text{if $\gamma<g-n$}
\end{array}\right.$$
and $$d_\ell=d_\ell(n,g,\gamma):=\left\{\begin{array}{l@{\quad\quad}l@{\quad\quad}l}n+g-1
&\text{if $\gamma\ge g-n$}\\
2n+\gamma &\text{if $\gamma<g-n$}
\end{array}\right.$$
and prove, among other things, the following three results.

\begin{theorem}\label{th21}{\em(\cite[Theorem 1.2]{BO})} Let $C$ be a curve of genus $g$ and Clifford index $\gamma$ and suppose $k>n$. Then $d_0(n,k)\ge d_\ell$.
\end{theorem}

\begin{theorem}\label{th22}{\em(\cite[Theorem 4.1]{BO})} Let $C$ be a curve of genus $g$ and Clifford index $\gamma$ and suppose that there exists a generated coherent system $(G,W)$ of type $(m,d,n+m)$ with $H^0(G^*)=0$ and $d_\ell\le d<d_u$. Then $U(n,d,n+m)\ne\emptyset$.
\end{theorem}

\begin{corollary}\label{cor8}{\em(\cite[Corollary 5.1]{BO})} Let $C$ be a smooth plane curve of degree $d\ge5$. Then $d_0(2,3)=d_\ell(2,g,\gamma)=d$ and $U(2,d,3)\ne\emptyset$. 
\end{corollary} 

In the corollary, one may note that the value of $d_0$ here is much smaller than the value for a general curve of the same genus given by Corollary \ref{cor4}. The paper \cite{BO} includes also examples where $d_0(n,n+1)=g-1$.

\section{Singular curves}\label{sec11}
 
The construction of moduli spaces for coherent systems can be generalized to arbitrary curves (see \cite{KN}). Of course, Teixidor's work depends on constructing coherent systems on reducible curves and in particular the concept of limit linear series (see \cite{Te4,Te5,Te1,Te2,Te6,Te3} and Teixidor's article in this volume), but apart from this there has been little work on coherent systems on singular curves. There is some work of Ballico proving the existence of $\alpha$-stable coherent systems for large $d$ \cite{Ba3,Ba6} and a paper of Ballico and Prantil for $C$ a singular curve of genus $1$ \cite{BP}.

The situation for nodal curves has now changed due to work of Bhosle, who has generalized results of \cite{BGN} and \cite{BGMN}. We refer to \cite{Bh1,Bh2} for details.

\section{Open Problems}\label{sec12}

There are many open problems in the theory of moduli spaces of coherent systems on algebraic curves. As in the main part of this article, we restrict attention here to problems of emptiness and non-emptiness. The basic outstanding problems can be formulated as follows.

\begin{problem}\label{p1}\begin{em} Given $C$ and $(n,k)$ with $n\ge2$, $k\ge1$, determine the value of $d_0(n,k)$ (the minimum value of $d$ for which $G(\alpha;n,d,k)\ne\emptyset$ for some $\alpha>0$).\end{em} 
\end{problem}

\begin{problem}\label{p2}\begin{em} Given $C$ and $(n,k)$, prove or disprove Model Theorem. If the theorem fails to hold, Theorem \ref{th1} implies that there are finitely many values of $d\ge d_0(n,k)$ for which $U(n,d,k)=\emptyset$. For each such $d$, find the range of $\alpha$ (possibly empty) for which $G(\alpha;n,d,k)\ne\emptyset$. Note that we already have a solution for this problem when $g=2$ and $k=n+1$; in this case (a) is true but not (b) or (c) (Theorem \ref{th9}). In Example \ref{ex}, (a) fails but (b) may possibly be true.
\end{em}\end{problem}

\begin{problem}\label{p3}\begin{em} For fixed genus $g$ and fixed $(n,d)$, we know that there exists a bound, independent of $\alpha$, on $h^0(E)$ for $E$ to be an $\alpha$-semistable coherent system. The bound given by \cite[Lemmas 10.2 and 10.3]{BGMN} looks weak. Find a better (even best possible) bound. An answer to this would be useful in estimating codimensions of flip loci.\end{em}
\end{problem}

We now turn to some more specific problems, which we state for a general curve, although versions for special curves could also be produced.

\begin{problem}\label{p4}\begin{em}
A conjecture of D.~C.~Butler \cite[Conjecture 2]{Bu} can be stated in the following form.

\begin{conj} {\bf (Butler's Conjecture)} Let $C$ be a general curve of genus $g\ge3$ and let $(E,V)$ be a general generated coherent system in $G_0(m,d,n+m)$. Then the dual span $D(E,V)$ belongs to $G_0(n,d,n+m)$.
\end{conj}

(Since $G_0(m,d,n+m)$ can be reducible, one needs to be careful over the meaning of the word ``general'' here.) There are many possible versions of this conjecture. For example, if one replaces $G_0$ by $G_L$, then the conjecture holds for all $(E,V)\in G_L(m,d,n+m)$ if $\gcd(m,n)=1$ and, even in the non-coprime case, a slightly weaker version holds \cite[section 5.4]{BGMN}. We seem, however, to be some distance from proving the general conjecture.\end{em}\end{problem}

\begin{problem}\label{p5}\begin{em}A slightly stronger version of Butler's conjecture is 

\begin{conj} {\bf (Butler's Conjecture (strong form))} Let $C$ be a general curve of genus $g\ge3$ and let $(E,V)$ be a general generated coherent system in $G_0(m,d,n+m)$. Then $D_V(E)$ is stable.
\end{conj}
When $m=1$ and $C$ is Petri, this conjecture holds for $V=H^0(E)$ (see \cite[Theorem 2]{Bu} and \cite[Proposition 4.1]{Br}). For general $V$, it is equivalent in this case to

\begin{conj}  Let $C$ be a Petri curve of genus $g\ge3$. If $\beta(n,d,n+1)\ge0$, then $U(n,d,n+1)\ne\emptyset$.
\end{conj}
(For the proof of this equivalence, see \cite[Conjecture 9.5 and Proposition 9.6]{BBN}.)
This result narrowly fails for $g=2$ and $d=2n$ (see Theorem \ref{th9}), but is proved in many cases in section \ref{sec6}.  For some further results on the strong form of Butler's Conjecture when $m=1$, see \cite{Mis}.
\end{em}\end{problem}

\begin{problem}\label{p6}\begin{em}
Let $C$ be a general curve of genus $g\ge3$. Calculate $d_0(2,4)$. For the current state of information on this problem, see section \ref{sec7}, especially Corollary \ref{cor5}.
\end{em}\end{problem}

\begin{problem}\label{p7}\begin{em}
Let $C$ be a general curve of genus $g$. Extend the results of section \ref{sec9} to $G(\alpha;2,K,k)$ for arbitrary $\alpha>0$ and hence to $U(2,K,k)$.
\end{em}\end{problem}

\begin{problem}\label{p8}\begin{em}
Related to Problem \ref{p7}, we can state an extended version of the conjecture of Bertram/Feinberg/Mukai.

\begin{conj}  Let $C$ be a general curve of genus $g$. Then
\begin{itemize}
\item $\beta(2,K,k)<0\Rightarrow G(\alpha;2,K,k)=\emptyset$ for all $\alpha>0$
\item $\beta(2,K,k)\ge0\Rightarrow U(2,K,k)\ne\emptyset$ for all $\alpha>0$.
\end{itemize}
\end{conj}
For the current state of information on this problem, see section \ref{sec9}.
\end{em}\end{problem}

Turning now to special curves, we give two problems.

\begin{problem}\label{p9}\begin{em}
Solve the basic problems for hyperelliptic curves. For Brill-Noether loci, a strong result is known \cite[section 6]{BMNO}. One could try to generalize this to coherent systems.
\end{em}\end{problem}

\begin{problem}\label{p10}\begin{em}
There are several possible notions of higher rank Clifford indices. Determine what values these indices can take and what effect they have on the geometry of the moduli spaces of coherent systems. What other invariants can be introduced (generalizing gonality, Clifford dimension, ...)?
\end{em}\end{problem}

We finish with two problems on a topic not touched on in this survey.

\begin{problem}\label{p11}\begin{em}
The moduli spaces can be constructed in finite characteristic (see \cite{KN}). Solve the basic problems in this case and compare (or contrast) them with those for characteristic $0$. Much of the basic theory should be unchanged, but detailed structure of the moduli spaces might well be different.
\end{em}\end{problem}

\begin{problem}\label{p12}\begin{em}
Obtain results for coherent systems on curves defined over finite fields. This would allow the possibility of computing zeta-functions and hence obtaining cohomological information in the characteristic $0$ case.
\end{em}\end{problem}

\end{document}